\documentclass[a4paper,12pt]{article}
\usepackage{geometry}
\usepackage[cp1251]{inputenc}
\usepackage[russian, ukrainian, english]{babel}
\usepackage{epstopdf}
\usepackage{graphicx}
\usepackage{amsmath,amssymb,euscript,amsthm}

\theoremstyle{plain}
\newtheorem{thm}{Theorem}

\theoremstyle{definition}

\theoremstyle{remark}
\newtheorem{remk}{Remark}
\newtheorem{prop}{Proposition}
\newtheorem{dfn}{Definition}
\numberwithin{equation}{section}
\numberwithin{thm}{section}
\numberwithin{lem}{section}
\numberwithin{nsl}{section}
\numberwithin{remk}{section}
\numberwithin{prop}{section}
\numberwithin{dfn}{section}
\begin{document}

\begin{center}
\textbf{D. KOROLIOUK, V.S. KOROLIUK}
\vskip15pt

\textbf{DIFFERENTIAL DIFFUSION MODEL WITH TWO EQUILIBRIUM STATES}
\end{center}
\vskip10pt


\vskip8pt
\section*{\normalsize Summary}
\hskip18pt The difference diffusion model with two equilibrium states is given by a stochastic equation with two components: the predicted one, which is determined by the regression function of increments with two equilibriums, and the stochastic one, which is the martingale difference.

We propose a classification of zones of influence of equilibriums according to asymptotic properties of trajectories of statistical experiments.

We study asymptotic behavior of statistical experiments, determined by the sums of $N$ sample values, as $N\to\infty$.

\vskip18pt

\section*{\normalsize Introduction}
\vskip8pt

In our works \cite{DK-6, DK-8} we research \textit{statistical experiments} (SE), given by the average sums of sample values that take binary values.

\begin{dfn} {Statistical experiments} (SE) are determined by average sums of sample values
$(\delta_n(k)$, $1\leq n\leq N)$, $k\geq0$, jointly independent at a fixed $k\geq0$ and equally distributed at different $n\in \, [1,N]$, taking a finite number of values, for simplicity two values $\pm1$:
\begin{equation}
\label{eq1}
S_N(k):=\frac{1}{N}\sum^N_{r=1}\delta_r(k) \ , \ \ -1\leq S_N(k)\leq1 \ , \ \  k\geq0.
\end{equation}
The corresponding SE frequencies are set by the average sums:
\begin{equation}
\label{eq2}
S_N^{\pm}(k):=\frac{1}{N}\sum_{n=1}^N\delta_n^\pm(k) \ , \ \ \delta_n^\pm(k):=I\{\delta_n(k)=\pm 1\} \ , \ \ k\geq0.
\end{equation}
Here, as always, indicator of a random event $I(A)=1$, if $A$ occurs, or $I(A)=0$, if $A$ does not happen.
The parameter $k\geq0$ means a sequence of stages of observations and is considered a discrete time that parameterizes the dynamics of statistical experiments.
\end{dfn}

\noindent The obvious condition of balance takes place:
\begin{equation}
\label{eq3}
S_N^+(k)+S_N^-(k)\equiv 1 \ , \ \ \forall k\geq0.
\end{equation}
The connection relations
\begin{equation}
\label{eq4}
S_N(k)=S_N^+(k)-S_N^-(k) \ , \ \ S_N^\pm(k)=\frac{1}{2}[1\pm S_N(k)],
\end{equation}
allow to study the evolution of SE, given by one of the processes defined (0.1) or (0.2).

In the sequel, the main attention is paid to the study of the dynamics of SE, determined by positive frequencies $S_N^+(k)$, $k\geq0$. However, it is also useful to study the dynamics of binary SE $S_N(k)$, $k\geq0$, given by the average sums (0.1).

\noindent First of all, the dynamics of CE is determined by \textit{evolutionary processes} (EP) $P_\pm(k)$, $k\geq0$, given by the conditional mathematical expectations:
\begin{equation}
P_\pm(k+1):= E[S_N^\pm(k+1)\,|\,S_N^\pm(k)=P_\pm(k)]  \ , \ \  k\geq0,
\end{equation}
and also
\begin{equation}
C(k+1):= E[S_N(k+1)\,|\,S_N(k)=C(k)]  \ , \ \  k\geq0.
\end{equation}
The relation (0.4) generates a connection between EPs (0.5) and (0.6):
\begin{equation}
C(k)=P_+(k)-P_-(k).
\end{equation}
Given the obvious condition of balance
\begin{equation}
P_+(k)+P_-(k)\equiv1  \ , \ \  \forall k\geq0,
\end{equation}
the frequency evolutionary processes $P_\pm(k)$, $k\geq0$, have a representation (cf. with (0.4)):
\begin{equation}
P_\pm(k)=\frac{1}{2}[1\pm C(k)]  \ , \ \  k\geq0.
\end{equation}
In Section 1, the difference evolution model is given by increments of frequency probabilities
\begin{equation}
\Delta P_\pm(k+1)=P_\pm(k+1)-P_\pm(k)  \ , \ \  k\geq0,
\end{equation}
as well as increments of binary EP
\begin{equation}
\Delta C(k+1)=C(k+1)-C(k)  \ , \ \  k\geq0,
\end{equation}
The dynamics of EP (0.5) and (0.6) is given by the conditional mathematical expectations
\begin{equation}
V^\pm(p_\pm):= E[\Delta S_N^\pm(k+1)\,|\,S_N^\pm(k)=p_\pm]  \ , \ \  0\leq p_\pm \leq 1,
\end{equation}
and
\begin{equation}
V_0(c):=-E[\Delta S_N(k+1)\,|\,S_N(k)=c]  \ , \ \  |c|\leq1.
\end{equation}
Note the obvious connection between conditional mathematical expectations (0.12) and (0.13):
\begin{equation}
V_0(c):=-[V^+(p_+)-V^-(p_-)]  \ , \ \  c=p_+-p_- \ , \ \ p_\pm=\frac{1}{2}[1\pm c].
\end{equation}
\vskip8pt
\begin{remk}
Conditional mathematical expectations (0.12) and (0.13) determine predictable components of CE (0.2) and (0.1).
\end{remk}
\vskip8pt
The main purpose of this work is to study the dynamics of SEs (0.1) and (0.2) and their evolutionary processes (0.5) and (0.6) by the stages $ k \ to \ infty $, taking into account the peculiarities of the representation of conditional mathematical expectations (0.12) and (0.13) . Unlike our previous work (see, e.g., \cite{DK-14}), we study the "zones of influence"\, of two equilibriums generated by the principle of "stimulation and restraint"\,. The universality of this principle is confirmed by real interpretations of the dynamics of economic processes \cite{Mas1, Mas2}, development of population genetics models \cite{KKR} and interpretation of learning processes \cite{KBK}.

\vskip8pt
\section{\normalsize The principle "stimulation and restraint"\,}
\vskip8pt
\noindent The frequency probabilities $P_+(k)$ of sample values $\delta_n^+(k)$, $k\geq1$, are determined by \textit{evolutionary processes} (EP), which dynamics by the stages $k\geq0$) is determined by \textit{regression function of increments} (RFI),which reflects the fundamental principle of interaction of collective behavior of a set of objects - the principle "stimulation and restraint"\, (see \cite{DK-14}).
\vskip8pt
\noindent From a mathematical point of view, the principle "stimulation and restraint"\, is given by the linear function of frequencies of values $p_\pm$ ($0\leq p_\pm\leq 1$, $p_++p_-=1$) of the sampling values $\delta_n^\pm(k)$:
\begin{equation*}
\pi_+p_+-\pi_-p_-
\end{equation*}
with two directing parameters $\pi_\pm$, which, without reducing the generality, also satisfy the properties of frequencies:
\begin{equation*}
0<\pi_\pm<1 \ , \ \ \pi_++\pi_-=1.
\end{equation*}
It follows that the principle "stimulation and restraint"\, is expressed by \textsl{frequency fluctuations} (FF)
\begin{equation*}
\pi_+p_+-\pi_-p_-=p_+-\pi_-.
\end{equation*}

The linear component of RFI can also be represented in the following form:
\begin{equation*}
\pi_-p_+-\pi_+p_-=p_+-\pi_+.
\end{equation*}
The product of two frequency fluctuations
\begin{equation}
(p_+-\pi_+)(p_+-\pi_-)=(p_--\pi_-)(p_--\pi_+)
\end{equation}
generates a nonlinear RFI with two equilibriums $\pi_\pm$. Completion of the construction of frequency RFI taking into account the boundary absorption states 0, 1 determines the increments of frequency probabilities:
\begin{equation}
V^\pm(p_\pm)=\mp V\cdot p_\pm(1-p_\pm)(p_\pm-\pi_\pm)(p_\pm-\pi_\mp) \ , \ \ V>0.
\end{equation}
There takes place a balance condition:
\begin{equation}
\label{eq3}
V^+(p_+)+V^-(p_-)\equiv 0 \ , \ \ p_++p_-=1.
\end{equation}

According to the principle of "stimulation and restraint"\,, linear RFI can also be as follows:
\begin{equation*}
\pi_\pm p_++\pi_\mp p_-=\pm\pi(p_+-\overline{\pi}_\mp) \ , \ \ \overline{\pi}_\mp:=\mp{\pi}_\mp/\pi
 \ , \ \ \pi:=\pi_+-\pi_-.
\end{equation*}
Without reducing the generality, we will consider $\pi:=\pi_+-\pi_->0$. Then the equilibriums
$\overline{\pi}_\mp=\mp{\pi}_\mp/\pi$ are outside the interval $(0,\,1)$:
\begin{equation*}
\overline{\pi}_-:=-\pi_-/\pi<0 \ , \ \ \overline{\pi}_+:=\pi_+/\pi>1.
\end{equation*}
Hence the equilibrium states determined by equilibriums $\overline{\pi}_\mp$, "stimulate"\, or "restrain"\, dynamics of frequency probabilities, without dividing the range of values - the interval $(0,\,1)$ on additional areas of influence.
Such equilibriums are also considered in Maslov's economic studies \cite{Mas1, Mas2}.

\vskip8pt
The presence of two equilibriums in a regression function of increments (RFI), which determine \textit{difference evolutionary equations} (DEE) for increments of frequency probabilities (0.5) and increments of binary evolutionary processes (0.6) significantly changes the dynamics of both evolutionary processes and the statistical experiments themselves.

Next we will consider frequency and binary evolutionary processes, which are characterized by "zones of influence"\, of equilibriums $\pi_\pm$, and determine attractive $\pi_+$ and repulsive $\pi_-$ by equilibriums (Theorem 2.1 in frequency representation and its binary variant - Theorem 2.2).

The interpretation of the zones of influence of two equilibriums is given, taking into account the asymptotic behavior of evolutionary processes by $k\to\infty$.

Further we investigate SEs with RFI, that has two equilibriums, according to the scheme of the works \cite{DK-6, DK-8}.

Initially (Section 4) the dynamics of SE (0.1) and (0.2) is given by the difference stochastic equation (4.3) and (4.4) taking into account the martingales (4.1) and (4.2) and their first two moments (4.5) - (4.7).

Then in Section 5, the classification of equilibriums of stochastic model SE is suggested using the method of stochastic approximation (\cite{Robb-Monr}).

Finally in Section 6, the stochastic component is approximated by a sequence distribution of independent normally distributed random variables. Theorem 6.1 gives the basis for the normal approximation of SE (Proposition 6.1).

In the final Section 7, the normalized SE are considered in discrete-continuous time and are approximated by a continuous Ornstein-Uhlenbeck process (\cite{DKor7, DKor9}).

\vskip8pt
\section{\normalsize Difference evolutionary model}
\vskip8pt
The increments of frequency probabilities (0.10) and binary EPs (0.11) are given by conditional mathematical expectations (0.12) and (0.13).
\begin{dfn}
The frequency probabilities $P_\pm(k)$, $k\geq0$, are given by solutions of \textit{difference evolution equations} (DEE)
\begin{align}
& \Delta P_\pm(k+1)=V^\pm(P_\pm(k)) \ , \ \ k\geq0,
\end{align}
and also
\begin{align}
& \Delta C(k+1)=-V_0(C(k)) \ , \ \ k\geq0,
\end{align}
\end{dfn}
The conditional mathematical expectations (2.1), (2.2) are determined by RFI
\begin{equation}
V^\pm(p_\pm)=\mp V\cdot p_\pm(1-p_\pm)(p_\pm-\pi_\pm)(p_\pm-\pi_\mp)  \ , \ \  0\leq p_\pm \leq 1,
\end{equation}
and also
\begin{equation}\begin{split}
V_0(c)=-V_0\cdot(&1-c^2)(c+\pi)(c-\pi),\\
&\pi=\pi_+-\pi_-  \ , \ \ V_0:=V/4 \ , \ \  |c|\leq1.
\end{split}\end{equation}
RFI (2.4) of binary EPs $C(k)$, $k\geq0$, also has two or two equilibriums $\mp\pi$.

Further analysis of the EP, given by DEE (2.1) and (2.2) taking into account RFI (2.3) and (2.4), is implemented according to the scheme of our works \cite{DK-6, DK-8}.

\begin{figure}[h]
\noindent\centering{
\includegraphics[width=170mm]{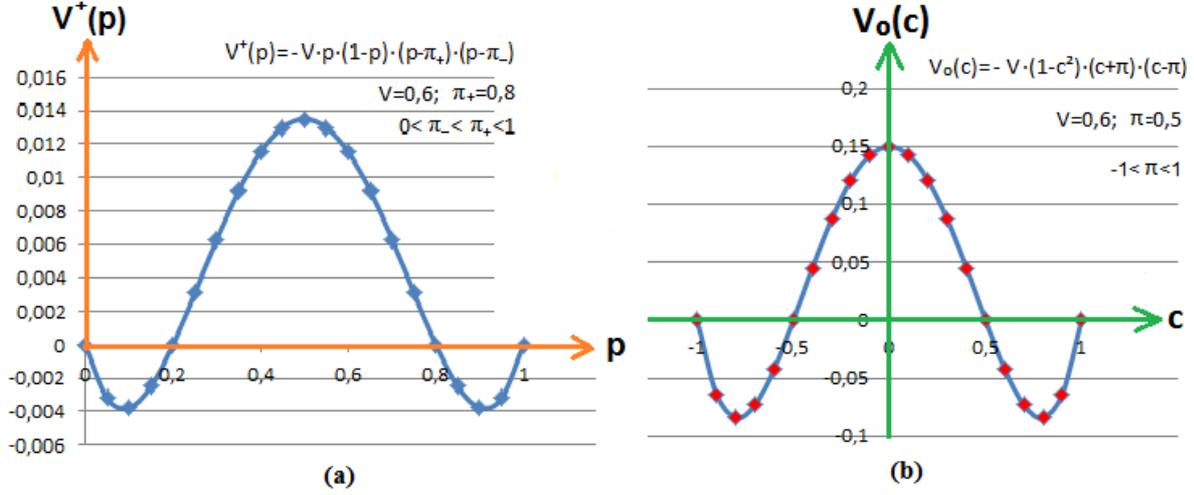}
}
\caption{Regression functions with two equilibriums: (a) frequency; (b) binary.}
\label{fig.1}
\end{figure}

The RFIs (2.3) and (2.4) allocate the zones set by equilibriums $\pi_\pm$ and  $\mp\pi$ (see Fig.1).

The equilibriums $\pi_\pm$ remarks three zones:
\begin{equation}
0<p_+<\pi_-  \ , \ \ \pi_-< p_+<\pi_+  \ , \ \  \pi_+< p_+< 1.
\end{equation}
Similarly, there are three zones that define binary RFI $V_0(c)$:
\begin{equation}
-1< c<-\pi  \ , \ \ -\pi< c<\pi  \ , \ \  \pi< c< 1.
\end{equation}
Now, using the classification of EP models, and taking into account the limit behavior of frequency probabilities $P_\pm(k)$, and binary EP $C(k)$ by  $k\to\infty$, we formulate the classification theorem.

\begin{thm} \label{thm2.1}
The frequency probabilities of alternatives $P_\pm(k)$, $k\geq0$, given by the solutions DSE (2.1), identify areas of influence (2.5) by the following asymptotic behavior:

$\spadesuit$ In the MA model (attractive):
\begin{equation}
\lim_{k\to\infty}P_\pm(k)=\pi_\pm  \ , \ \  \pi_-< P_\pm(0)\leq 1.
\end{equation}

$\spadesuit$ In the MR model (repulsive):
\begin{align}
& \lim_{k\to\infty}P_+(k)= 0 \ , \ \ 0< P_+(0)<\pi_-; \\
& \lim_{k\to\infty}P_-(k)= 1  \ , \ \  \pi_+< P_-(0)< 1.
\end{align}
\end{thm}

\vskip8pt
Similarly, binary EPs are determined by the zones of influence (2.6) of equilibriums $\mp\pi$ by asymptotic behavior:
\begin{thm}\label{thm2.2}
Binary EP $C(k)$, $k\geq0$, given by the solutions of DEE (2.2), determine the zones of influence of equilibriums $\mp\pi$:

$\spadesuit$ In the MA model (attractive):
\begin{equation}
\lim_{k\to\infty}C(k)=\pi  \ , \ \  -\pi< C(0)< 1;
\end{equation}

$\spadesuit$ In the MR model (repulsive):
\begin{equation}
\lim_{k\to\infty}C(k)=0 \ , \ \  0< C(0)< -\pi.
\end{equation}
\end{thm}
\vskip8pt
\begin{proof}[Proof of theorems \ref{thm2.1}, \ref{thm2.2} are similar] Therefore, we consider only the frequency positive probabilities of alternatives $P_\pm(k)$ $k\geq0$ with the regression function of increments (2.3) (see also Fig. 1a):
\begin{equation}
V^+(p_+)= -V\cdot p_+(1-p_+)(p_+-\pi_-)(p_+-\pi_+),
\end{equation}
takes the following values:
\begin{equation}\begin{split}
& V^+(P_+(k))<0 \ , \ \ 0< P_+(0)<\pi_- \ \text{or} \
\pi_+< P_+(0)< 1; \\
& V^+(P_+(k))>0 \ , \ \ \pi_-< P_+(0)<\pi_+.
\end{split}
\end{equation}
Monotonicity of increments (2.13) in zones of influence of equilibriums and the boundess of probabilities $P_+(k) \ : \ \ 0\leq P_+(k)\leq1$, allow us to conclude that there exist the limits (2.7) and (2.8) for the frequencies
\begin{equation}
p_+^*=\lim_{k\to\infty}P_+(k).
\end{equation}
The value of the boundaries (2.7) and (2.8) follows from the obvious equation
\begin{equation}
0=V^+(p_+^*)= p_+^*(1-p_+^*)(p_+^*-\pi_-)(p_+^*-\pi_+).
\end{equation}
For the initial value in the area of influence $\pi_-< P_+(0)<\pi_+$, the monotonic growth of increments (2.13) means that $\lim_{k\to\infty}P_+(k)=p_+^*=\pi_+$.

Similarly, the limit behavior of positive frequency probabilities $P_+(k)$, $k\geq0$, is determined in the areas of influence $\pi_+< P_+(0)< 1$ and the condition $0< P_+(0)<\pi_-$, is provided by a monotonous decline in increments.
\end{proof}

Classification of zones of influence of equilibriums of RFI allows to make the preliminary conclusion concerning real interpretation of equilibriums in modern economic space, and also in models of population genetics \cite{KKR}, or in learning models  \cite{KBK}.

\vskip8pt
\section{\normalsize Interpretation of zones of influence of equilibriums $\pi_\pm$ in the economic space}
\vskip8pt
The dynamics of frequency probabilities $P_+(k)$, $k\geq0$, in the area of attractive equilibrium $\pi_+$
\begin{equation}
\lim_{k\to\infty}P_+(k)=\pi_+ \ , \ \ P_+(0)>\pi_+
\end{equation}
means that the presence of "relative"\, capital $CP_+(k)>C\pi_+$ leads to a general reduction of capital to equilibrium state $C\pi_+$, which characterizes the objectivity of the principle of "stimulation and restraint"\,.

However, attracting equilibrium $\pi_+$  in zone \ $\pi_-<P_+(0)<\pi_+$ \ also leads to an increase in capital to equilibrium \ $C\pi_+$, \ which also characterizes the objectivity of the principle of "stimulation and restraint"\,.

Finally, in the zone of repulsive equilibrium, the dynamics of "relative"\, \ capital
\begin{equation}
\lim_{k\to\infty}P_+(k)=0 \ , \ \ 0<P_+(0)<\pi_-
\end{equation}
means that the subject of economic space with "relative"\, \ capital less than \ $C\pi_-$, \ lose capital.

There is a strategic problem: the reduction of the zone of "impoverishment"\, \ subjects of the economic space, and ideally - the elimination of the zone \ $(0,\,\pi_-)$.

\vskip8pt
\section{\normalsize Differential stochastic model}
\vskip8pt

The difference evolutionary equations (2.1), (2.2) together with RFI (2.3) - (2.4),  generate deterministic evolutionary processes, expressed by the expected characteristics of SE.

The dynamics of CE (0.1) or (0.2) is set by \textsl{difference stochastic equation} (DSE) using martingales (by $k\geq0$)
\begin{equation}
\Delta\mu_N^\pm(k+1):=\Delta S_N^\pm(k+1)-V^\pm(S_N^\pm(k)) \ , \ \ k\geq0,
\end{equation}
and also
\begin{equation}
\Delta\mu_N(k+1):=\Delta S_N(k+1)+V_0(S_N(k)) \ , \ \ k\geq0.
\end{equation}

\begin{dfn}
The frequency statistical experiments (0.2) are given by DSE solutions
\begin{equation}
\Delta S_N^\pm(k+1)=V^\pm(S_N^\pm(k))+\Delta\mu_N^\pm(k+1) \ , \ \ k\geq0,
\end{equation}
Binary statistical experiments (0.1) are given by the DSE solution
\begin{equation}
\Delta S_N(k+1)=-V_0(S_N(k))+\Delta\mu_N(k+1) \ , \ \ k\geq0.
\end{equation}
\end{dfn}

The martingales (4.1) and (4.2) are characterized by the first two points:
\begin{align}
& E[\Delta\mu_N^\pm(k+1)]=0 \ , \ \ E[\Delta\mu_N(k+1)]=0 \ , \ \ k\geq0, \\
& E[(\Delta\mu_N^\pm(k+1))^2\,|\,S_N^\pm(k)]=V_+(S_N^+(k))\cdot V_-(S_N^-(k))/N \ , \ \ k\geq0, \\
& E[(\Delta\mu_N(k+1))^2\,|\,S_N(k)]=[1-V^2(S_N(k))]/N \ , \ \ k\geq0.
\end{align}
Here the regression functions
\begin{align}
&V_\pm(p_\pm):=p_\pm\mp V^\pm(p_\pm) \ , \ \ 0\leq p_\pm\leq1, \\
&V(c):=c-V_0(c) \ , \ \ |c|\leq1.
\end{align}
\vskip8pt
The calculation of the second moments of the martingales (4.1) and (4.2) uses the following obvious relations:
\begin{equation}\begin{split}
E[(\delta_n^\pm(k+1))^2\,&|\,S_N^\pm(k)=p_\pm]=E[\delta_n^\pm(k+1)\,|\,S_N^\pm(k)=p_\pm]\\
&=E[S_N^\pm(k+1)\,|\,S_N^\pm(k)=p_\pm]=V_\pm(p_\pm) \ , \ \ k\geq0, \\
\end{split}
\end{equation}
and also
\begin{equation}\begin{split}
&E[(\delta_n(k+1))^2]\equiv 1 \ , \ \ \forall k\geq0, \\
&E[ \delta_n(k+1)\,|\,S_N(k)=c]=E[S_N(k+1)\,|\,S_N(k)=c]=V(c) \ , \ \ |c|\leq 1.
\end{split}
\end{equation}
So we have
\begin{equation}\begin{split}
D[\delta_n(k+1)\,|\,S_N(k)=c]:=&E[(\delta_n(k+1))^2\,|\,S_N(k)=c]-\\
&-[E[\delta_n(k+1)\,|\,S_N(k)=c]]^2=1-V^2(c).
\end{split}
\end{equation}
Similarly
\begin{equation}
D[\delta_n^\pm(k+1)\,|\,S_N^\pm(k)=p_\pm]=V_\pm(p_\pm)-V_\pm^2(p_\pm)=V_+(p_+)V_-(p_-).
\end{equation}
In addition, the following relations are used
\begin{equation}
D[S_N(k+1)\,|\,S_N(k)=c]=\frac{1}{N}D[\delta_n(k+1)\,|\,S_N(k)=c]=[1-V^2(c)]/N,
\end{equation}
which explains the presence of the multiplier $N^{-1}$ in formulas (4.6) and (4.7).

It remains to note that by the definition of martingales (4.1), (4.2) their conditional variances are determined by the equality
\begin{equation}
E[(\Delta\mu_N(k+1))^2\,|\,S_N(k)]=D[S_N(k+1)\,|\,S_N(k)].
\end{equation}
Really from definition (4.2) we have
\begin{equation}\begin{split}
&E[(\Delta\mu_N(k+1))^2\,|\,S_N(k)]=E[[\Delta S_N(k+1)+V_0(S_N(k))]^2\,|\,S_N(k)]\\
&=E[S_N^2(k+1)\,|\,S_N(k)]-[E[S_N(k+1)\,|\,S_N(k)]]^2=D[S_N(k+1)\,|\,S_N(k)].
                \end{split}
\end{equation}
\vskip8pt
\section{\normalsize Classification of equilibriums of the stochastic model SE}
\vskip8pt

The presence of the stochastic component $\Delta\mu_N(k+1)$ DSE (4.4) for binary SE significantly complicates the solution of the problem of SE stability. It is proposed to use the method \emph{stochastic approximation} (SA), developed by Robinson and Monroe \cite{Robb-Monr} to find the root of the regression equation with stochastic errors.

We enter the SA parameters
\begin{equation*}
a_k:=a/k \ , \ \ k\geq1,
\end{equation*}
satisfying the approximation conditions:
\begin{equation}
\sum_{k=1}^{\infty}a_k=+\infty \ , \ \ \sum_{k=1}^{\infty}a^2_k=a\sum_{k=1}^{\infty}\frac{1}{k^2}=a\frac{\pi^2}{6}<\infty \ , \ \ k\geq0.
\end{equation}
The SA procedure for DSE (4.4) is expressed by the scheme
\begin{equation}
\Delta\alpha_N(k+1)=a_k[-V_0(\alpha_N(k))+\sigma(\alpha_N(k))\Delta\mu_N(k+1)]
 \ , \ \ k\geq1.
\end{equation}
The SA normalization parameters satisfy the conditions (5.1).

\begin{remk}
\emph{The discrete Markov diffusion} (DMD) $\alpha_N(k)$, $k\geq0$, given by the solution of DSE (5.2), is characterized by the first two points:
  \begin{equation}
\begin{split}
&E[\Delta\alpha_N(k+1)\,|\,\alpha_N(k)=c]=-a_kV_0(c) \ , \ \ k\geq0,\\
&E\Bigl[[\Delta\alpha_N(k+1)]^2\,\Bigl|\,\alpha_N(k)=c\Bigr]=a_k^2\bigl[\sigma^2(c)+V_0^2(c)/N\bigr].
\end{split}
   \end{equation}
\end{remk}
\vskip8pt

To classify the zones of influence of equilibriums $\pm\pi$ ($|\pi|<1$)  introduce classifiers of equilibriums:
   \begin{equation}
\Lambda_\pm(c):=-V_0(1-c^2)(c\mp\pi) \ , \ \ |c|\leq1.
   \end{equation}
The following property of classifiers (5.4) is used:
   \begin{equation}
\Lambda_+(c)<0\ , \ \ -\pi<c<1 \ \ ; \ \ \  \Lambda_-(c)<0\ , \ \ -1<c<-\pi.
   \end{equation}

The SA procedure uses Theorems 2.7.2 by Nevelson-Hasminski \cite{NE-HA1973}, adapted to the SE models:
\begin{thm} (\cite[Thm. 2.7.2]{NE-HA1973}).
\label{thm3.1}
Suppose there is a non-negative function $\Lambda_0(c)$, $|c|\leq1$, having an equilibrium point $c_0$: $\Lambda_0(c_0)=0$, and satisfies the inequalities around the point $c_0$, by a fixed $h>0$:
   \begin{equation}
\sup_{\frac{1}{h}\leq|c-c_0|\leq h}\Lambda_0(c)(c-c_0)<0.
   \end{equation}
Then the procedure SA $\alpha_N(k)$, $k\geq0$, converges, with probability 1, to the equilibrium point $c_0$:
 \begin{equation}
   \alpha_N(k)\xrightarrow{P1} c_0 \ , \ \ k\to\infty.
   \end{equation}
\end{thm}

Now the classification of zones of influence of equilibriums $\pm\pi$ is formulated as follows:
\begin{thm}
\label{thm3.2}
The dynamics of SE $S_N(k)$, $k\geq0$, is characterized by classifiers (5.4):\\
I: In the area of influence of the attractive equilibrium $\pi$:
 \begin{equation}
P\lim_{k\to\infty}\alpha_N(k)=\pi \ , \ \ -\pi<\alpha_N(0)<1.
   \end{equation}
II: I: In the area of impact of the repulsive equilibrium $-\pi$:
 \begin{equation}
P\lim_{k\to\infty}\alpha_N(k)=0 \ , \ \ -1<\alpha_N(0)<-\pi.
   \end{equation}

\end{thm}

\begin{proof}[Proof of the theorem \ref{thm3.2}]
The classifiers (5.4) and the Nevelson-Hasminski theorem are used \ref{thm3.1}.

Calculating the value of generator DMD $\alpha_N(k)$, $k\geq0$, on test functions $\varphi_\pm(c)=(c\pm\pi)^2$ gives the following
\begin{equation}\begin{split}
L_\alpha\varphi_\pm(c):&=E[\varphi_\pm(c+\Delta\alpha_N(k+1))-\varphi_\pm(c)\,|\,\alpha_N(k)=c]\\
&=2E[\Delta\alpha_N(k+1)\,|\,\alpha_N(k)=c](c\pm\pi)+E[(\Delta\alpha_N(k+1))^2\,|\,\alpha_N(k)=c],
\end{split}\end{equation}
that is, we have
\begin{equation}
L_\alpha\varphi_\pm(c)=-2a_k\Lambda_\pm(c)(c\mp\pi)^2+a_k^2B_N(c).
\end{equation}
Taking into account the properties of classifiers (5.8) - (5.9), the Nevelson-Hasminski theorem substantiates the statement of the theorem \ref{thm3.2}.
\end{proof}

\vskip8pt
\section{\normalsize Approximation of the stochastic component}
\vskip8pt

The martingale property (by $k\geq0$) of stochastic components (4.1) and (4.2)
allows to characterize them in more detail. Namely, the normalized martingales (4.1) and (4.2) are determined by Bernoulli distributions of possible frequency values $\nu_N^\pm$ at a given distribution $p_\pm$:\\
$\clubsuit$ in frequency representation
\begin{equation}
\begin{split}
&P\{\,\Delta\mu_N^\pm(p_\pm)=\nu_N^\pm/N\,|\,S_N^\pm(k)=p_\pm\,\} \\
   &=\frac{N!}{\nu_N^+!\nu_N^-!}(p_+)^{\nu_N^+}(p_-)^{\nu_N^-} \ , \ \ \nu_N^++\nu_N^-= N \ , \ \ 0\leq\nu_N^\pm\leq N.
\end{split}
\end{equation}
$\clubsuit$  in binary representation
\begin{equation}
\begin{split}
&P\{\,\Delta\mu_N(c)=\nu_N/N\,|\,S_N(k)=c\,\} \\
&=\frac{N!}{\nu_N^+!\nu_N^-!}(p_+)^{\nu_N^+}(p_-)^{\nu_N^-} \ , \ \ \nu_N=\nu_N^+-\nu_N^- \ , \ \ -N\leq\nu_N\leq N.
\end{split}
\end{equation}
The martingale property of stochastic components (4.1) and (4.2) as well as their Bernoulli distributions (6.1) and (6.2) provide a "conditional form"\, of approximation using evolutionary processes (0.12) and (0.13).
\vskip8pt

\begin{thm}
\label{thm4.1}
The normalized stochastic components (6.1) and (6.2) are approximated, in distribution, at $N\to\infty$, by a sequence of independent, normally distributed random variables
$W_{k+1}^\pm$, $W_{k+1}$, $k\geq0$:
\begin{align}
&\sqrt{N}\Delta\mu_N^\pm (p_\pm) \xrightarrow{D} W_{k+1}^\pm \ , \ \ E[W_{k+1}^\pm ]^2=\breve{\sigma}^2(p_+,p_-);\\
&\sqrt{N}\Delta\mu_N (c)  \xrightarrow{D} W_{k+1} \ , \ \ E[W_{k+1}]^2=\sigma^2(c),
\end{align}
which are characterized by the second moments
\begin{equation}
\breve{\sigma}^2(p_+,p_-)=V_+(p_+)\cdot V_-(p_-) \ , \ \ \sigma^2(c)=1-V^2(c),
\end{equation}
and are determined by regression functions
\begin{equation}
V_\pm(p_\pm):=p_\pm\mp V^\pm(p_\pm) \ , \ \ V(c):=c-V_0(c).
\end{equation}
\end{thm}

\begin{proof}[The proof of the theorem \ref{thm4.1}] The Bernoulli distributions (6.1) and (6.2) of normalized matringals (4.1) and (4.2) in "conditional form"\, are generated by normalized sums of sample values
\begin{align}
&\Delta\mu_N^\pm(p_\pm)=\frac{1}{\sqrt{N}}\sum_{n=1}^{N}\beta_n^\pm(p_\pm)
\ , \ \ \beta_n^\pm(p_\pm)=\delta_n^\pm(k+1)-V_\pm(p_\pm)\\
&\Delta\mu_N(c)=\frac{1}{\sqrt{N}}\sum_{n=1}^{N}\beta_n(c)
\ , \ \ \beta_n(c)=\delta_n(k+1)-V(c).
\end{align}
Taking into account the definition of frequency martignals (4.1) and frequency regression functions (6.6), the frequency random variables in the sum (6.7) have the representation:
\begin{equation}
\beta_n^\pm(p_\pm)=
\begin{cases}
1-V_\pm(p_\pm), & \mbox{with probability } P_\pm=V_\pm(p_\pm), \\
\hskip12pt -V_\pm(p_\pm), & \mbox{with probability } P_\mp=V_\mp(p_\mp).
\end{cases}
\end{equation}
Similarly, taking into account the definitions of binary martignals (4.2) and binary regression functions (6.6), binary random variables in the sum (6.8) have the representation:
\begin{equation}
\beta_n(c)=
\begin{cases}
\ \ \ 1-V(c), & \mbox{with probability } P_+(c)=\frac{1}{2}[1+V(c)], \\
-(1+V(c)), & \mbox{with probability } P_-(c)=\frac{1}{2}[1-V(c)].
\end{cases}
\end{equation}
\vskip8pt
The sample values (6.9) and (6.10) are characterized by the first two points:
\begin{align}
& E\beta_n^\pm=E\beta_n=0, \\
& E(\beta_n^\pm)^2=\breve{\sigma}^2(p_+,p_-)=V_+(p_+)\, V_-(p_-) \ , \ \ E(\beta_n)^2=\sigma^2(c)=1-V^2(c).
\end{align}
Therefore, the normalized martingales (6.3) - (6.4) are also characterized by the first two points:
\begin{align}
& E\Delta\mu_N^\pm(p_\pm)=E\Delta\mu_N(c)=0, \\
& E(\Delta\mu_N^\pm(p_\pm))^2=\breve{\sigma}^2(p_+,p_-)=V_+(p_+)\, V_-(p_-),\\
& E(\Delta\mu_N(c))^2=\sigma^2(c)=1-V^2(c).
\end{align}
The use of central limit theorem for sums of independent, equally distributed random variables \cite[\S 3]{Shi-1} completes the proof \ref{thm4.1}.
\end{proof}
\vskip8pt
\begin{prop}
SE with normal approximation of the stochastic component is given by the following DSE (\cite{DK-18}, \cite{DK-19}):
\begin{align}
& \Delta\zeta_N^\pm(k+1)=V^\pm(\zeta_N^\pm(k))
                        +\frac{1}{\sqrt{N}}\breve{\sigma}(\zeta_N^\pm(k))W^\pm(k+1),  \\
& \Delta\zeta_N(k+1)=-V_0(\zeta_N(k))
                        +\frac{1}{\sqrt{N}}\sigma(\zeta_N(k))W(k+1)
\end{align}
Here $W^\pm(k+1)$, $W(k+1)$, $k\geq0$ are sequences of normally distributed standard random variables; the parameters $\breve{\sigma}^2$, ${\sigma}^2$ are determined in (6.14) - (6.15).
\end{prop}

\vskip8pt
\section{\normalsize Approximation of SE in discrete - continuous time}
\vskip8pt

The approximation of the normalized stochastic component, which is established in Theorem 6.1, gives rise to a new problem of SE approximation in discrete - continuous time.

\begin{dfn} DSE in discrete - continuous time $k=[Nt]$,  $t\geq0$, is given by two normalized components:
\begin{equation}
\Delta\zeta_N(t+1/N)=-\frac{1}{N}V_0(\zeta_N(t))
                        +\frac{1}{\sqrt{N}}\sigma(\zeta_N(t))\Delta\mu_N(t+1/N) \ , \ \ t\geq0.
\end{equation}
The increments of SE in discrete - continuous time are determined by the equality:
\begin{equation*}
\Delta\zeta_N(t+1/N):=\zeta_N(t+1/N)-\zeta_N(t) \ , \ \ t\geq0.
\end{equation*}
\end{dfn}
The normalized stochastic martingale component is given by the conditional Bernoulli distribution (cf. (2.12)).
\begin{equation}
\begin{split}
&P\Bigl\{\frac{1}{\sqrt{N}}\Delta\mu_N(k+1)=\nu_N/N)\,\Bigr|\,S_N(k)=C(k)\Bigr\}\\
   & =\frac{N!}{\nu_N^+!\nu_N^-!}(p_+)^{\nu_N^+}(p_-)^{\nu_N^-} \ , \ \ \nu_N:=\nu_N^+-\nu_N^-.
\end{split}
\end{equation}
Here, by definition (see item 2)
\begin{equation}
P_\pm(k):=\frac{1}{2}[1\pm C(k)] \ , \ \ k=[Nt] \ , \ \ t>0.
\end{equation}
So the first two moments have the following values
\begin{equation}\begin{split}
& E\Delta\mu_N(t+1/N)\equiv0 \ , \ \ \forall t\geq0,\\
& E[(\Delta\mu_N(t+1/N))^2\,|\,\zeta_N(t)=C(k)]=1-V^2(C(k)) \ , \ \ k=[Nt].
\end{split}\end{equation}
So we obtain the approximation theorem
\vskip8pt

\begin{thm}\label{thm6.1}
The normalized SEs given by the DSE solution (7.1) are approximated by the Ornstein-Uhlenbeck process of the type given by the solution of the differential stochastic equation
\begin{equation}
d\zeta_N(t)=-V_0(\zeta(t))dt
                        +\sigma(\zeta(t))dW(t) \ , \ \ t\geq0,
\end{equation}
characterized by a generator
\begin{equation}
L_\zeta\varphi(c) =-V_0(c)\varphi'(c)+\frac{1}{2}\sigma^2(c)\varphi''(c)
                                                   \ , \ \ |c|\leq1.
\end{equation}
\end{thm}
\begin{proof}[Proof of the theorem \ref{thm6.1}] The martingale characterization of Markov process is used $\zeta_N(t)$, $t\geq0$, given by the solution of DSE (7.1):
\begin{equation}
\mu_N(t):=\varphi(\zeta_N(t))-\varphi(\zeta_N(0))-\int_{0}^{N[t/N]}
L_\zeta^{(N)}\varphi(c)(\zeta_N(s))ds
                                                   \ , \ \ k=[Nt].
\end{equation}
The normalized generator is set as follows:
\begin{equation}
L_\zeta^N\varphi(c) =NE\Bigl[\varphi(c+\Delta\zeta_N(t))-\varphi(c)\,|\,\zeta_N(t)=c\Bigr].
\end{equation}
It is not difficult to make sure that the right-hand side of equation (7.7) taking into account (7.8) defines the martingale:
\begin{equation}
E\mu_N(t)=0 \ , \ \ E[\mu_N(t+s)\,|\,\mu_N(t)]=\mu_N(t) \ , \ \ s>0 \ , \ \ t\geq0.
\end{equation}

Next, we use an approximation of the generator (7.8) on fairly smooth numerical test functions
$\varphi(c)\in C_B^3(R)$, bounded and three times continuously differentiated with limited derivatives
\begin{equation}
L_\zeta^N\varphi(c) = L_\zeta\varphi(c)+R_N\varphi(c),
\end{equation}
with residual term:
\begin{equation}
R_N\varphi(c) \ \longrightarrow 0 \ , \ \ N\to\infty \ , \ \ \varphi(c)\in C_B^3(R)
\end{equation}
Approximation (7.10) - (7.11) substantiates the convergence of Markov process generators
$\zeta_N(t)$, $t\geq0$, in series scheme with parameter $N\to\infty$, to the generator (7.6) of the Ornstein-Uhlenbeck process type (7.5), which ensures the convergence of finite-dimensional distributions of normalized SE
\begin{equation}
\zeta_N(t) \Rightarrow \zeta(t) \ , \ \ N\to\infty,
\end{equation}
under the additional condition of the convergence of initial states
\begin{equation}
\zeta_N(0) \Rightarrow \zeta(0) \ , \ \ N\to\infty.
\end{equation}
Theorem \ref{thm6.1} is proved.

\end{proof}

\vskip8pt
\section*{\normalsize Conclusions}
\

The statistical experiments determined by linear regression functions of increments, which reflect the basic principle of interaction "stimulation - restraint"\,  are investigated.

It turns out that two linear growth regression functions have equilibriums in the interaction zone. The main attention is paid to statistical experiments with two equilibriums generated by the product of two linear regression functions of increments, with one of the equilibriums being attractive and the other repulsive.

Such a scheme of constructing the regression function of increments generates three zones of influence with two equilibriums.

\end{document}